\documentclass[11pt]{amsart}
\usepackage[a4paper,margin=2.5cm]{geometry}
\usepackage{amsmath,amssymb,amsthm,mathtools,enumitem,hyperref}
\hypersetup{hidelinks}
\newtheorem{theorem}{Theorem}[section]
\newtheorem{proposition}[theorem]{Proposition}
\newtheorem{lemma}[theorem]{Lemma}
\newtheorem{corollary}[theorem]{Corollary}
\newtheorem{conjecture}[theorem]{Conjecture}
\theoremstyle{remark}\newtheorem{remark}[theorem]{Remark}
\numberwithin{equation}{section}
\newcommand{\norm}[1]{\left\lVert#1\right\rVert}
\newcommand{\R}{\mathbb R}\newcommand{\C}{\mathbb C}
\newcommand{\Z}{\mathbb Z}\newcommand{\N}{\mathbb N}
\DeclareMathOperator{\spr}{\rho}
\title[Mean size of Daubechies wavelet packets]{Mean size and Schauder-basis properties\\of Daubechies wavelet packets}
\author[M.\ Nielsen]{Morten Nielsen}
\address{Department of Mathematical Sciences\\ Aalborg University\\
Thomas Manns Vej 23\\ DK-9220 Aalborg East\\ Denmark}
\email{mnielsen@math.aau.dk}
\date{}
\begin{document}
\begin{abstract}
In 2002 Nielsen and Zhou conjectured that the shifted wavelet packet
system associated with a Daubechies filter of length at least four
fails to be a Schauder basis of $L^p(\R)$ for $1\le p\le\infty$,
$p\ne2$, and that the packets are not uniformly bounded in $p$-mean
across scales for any $p>2$. The basis obstruction throughout the range
$1<p<\infty$, $p\ne2$, was previously established for length four,
via an explicit computation; the full $p$-mean assertion remained open. We prove both for every Daubechies filter of length at least
four: the high-pass transition matrices satisfy $\rho_p\rho_{p'}>2$ for
$1<p<\infty$, $p\ne2$, and for $p>2$ the full family of transition
matrices has $p$-norm joint spectral radius exceeding the Haar value
$4^{1/p}$. The key step is structural: an affine piece of the matrix
pressure forces constant spectral radius, which by a theorem of
Protasov and Voynov forces simultaneous orthogonality of the transition
matrices, incompatible with the double zero of the high-pass symbol.
The results extend to every real spectral factor with the Daubechies
magnitude response, including the least-asymmetric filters.
\end{abstract}
\maketitle

\section{Introduction}\label{sec:intro}

\subsection{Wavelet packets}
Let $m\ge2$, $d=2m-1$, and let $a=(a_0,\ldots,a_d)$ be the real
Daubechies low-pass filter of length $2m$ with $m$ vanishing moments
\cite{D88}, normalized by
\begin{equation}\label{eq:lowpass}
 \sum_k a_k=2,\qquad \sum_k a_k a_{k+2j}=2\delta_{j0}.
\end{equation}
Coefficients outside $\{0,\ldots,d\}$ are zero. Let $\phi$ be its
scaling function, so that
\[
 \phi(x)=\sum_k a_k\phi(2x-k).
\]
We use the standard facts that $\phi$ is bounded and compactly
supported and that its integer translates are orthonormal
\cite{D88}. Set $b_k=(-1)^k a_{1-k}$, and define
\begin{equation}\label{eq:recursion}
 w_0=\phi,\qquad
 w_{2j}(x)=\sum_k a_k w_j(2x-k),\qquad
 w_{2j+1}(x)=\sum_k b_k w_j(2x-k),\quad j\ge0.
\end{equation}
For $j=0$, the even recursion reproduces the refinement equation for
$w_0=\phi$. The shifted packet system
\[
 \mathcal W=\{w_j(\cdot-k):j\ge0,\ k\in\Z\}
\]
is an orthonormal basis of $L^2(\R)$ \cite{NZ}. Every packet is
bounded and supported in an interval of length at most $d$ (the
support lengths $\ell_n$ at depth $n$ satisfy $\ell_0=d$ and
$\ell_{n+1}\le(\ell_n+d)/2$), hence belongs to every $L^r(\R)$,
$1\le r\le\infty$. At depth $n$ the recursion produces the $2^n$
functions $w_0,\ldots,w_{2^n-1}$; the all-high-pass branch is
$w_{2^n-1}$.

Put $h_k=(-1)^ka_k$, with $h_k=0$ outside $\{0,\ldots,d\}$, and define
the $d\times d$ high-pass transition matrices
\begin{equation}\label{eq:matrices}
 (B_\eta)_{\alpha\beta}=h_{\eta+2\alpha-\beta},
 \qquad 0\leq\alpha,\beta\leq d-1,\quad\eta=0,1,
 \qquad \mathcal B=(B_0,B_1).
\end{equation}
This is the pair $C^1$ of \cite[Section~3.1 and Example~2]{NZ}, with the
index convention $\eta+2\alpha-\beta$ of \cite[(2.4)]{NZ}; the mask $h$
is a reflected translate of $b$ up to sign, see Section~\ref{sec:transfer}.
Let $\mathcal T$ denote the family of four transition matrices of the
low- and high-pass masks, defined in \eqref{eq:fourfamily} below.

\subsection{Mean size and the conjecture}
For a finite family $\mathcal A=(A_1,\ldots,A_\ell)$ of square matrices
and $1\le p<\infty$, the $p$-norm joint spectral radius is
\[
 \rho_p(\mathcal A)=\lim_{n\to\infty}
 \Bigl(\sum_{|I|=n}\norm{A_I}^p\Bigr)^{1/(np)},
\]
where $A_I=A_{i_n}\cdots A_{i_1}$ for $I=i_1\cdots i_n$. Nielsen and
Zhou \cite[Theorem~2 and Corollary~3.1]{NZ} expressed the growth of
wavelet packets through such radii: for $1<p<\infty$,
\begin{align}
 \lim_{n\to\infty}\norm{w_{2^n-1}}_p^{1/n}
 &=2^{-1/p}\rho_p(\mathcal B),\label{eq:NZ-branch}\\
 M_p:=\lim_{n\to\infty}
 \Bigl(\sum_{j=0}^{2^n-1}\norm{w_j}_p^p\Bigr)^{1/(np)}
 &=2^{-1/p}\rho_p(\mathcal T).\label{eq:NZ-mean}
\end{align}
The quantity $M_p$ is the \emph{mean size} of the packet tree. For the
Haar filter every packet is a Walsh function of modulus one on $[0,1]$,
and $\rho_p(\mathcal T)=4^{1/p}$ \cite[Example~1]{NZ}; the Haar packets
are thus uniformly bounded in $p$-mean. By a duality argument,
\cite[Lemma~3.1 and Remark~3.3]{NZ} showed that if $\mathcal W$ is a
Schauder basis of $L^p(\R)$, then
\begin{equation}\label{eq:NZ-condition}
 \rho_p(\mathcal B)\rho_{p'}(\mathcal B)=2,\qquad \frac1p+\frac1{p'}=1,
\end{equation}
while $\rho_p(\mathcal B)\rho_{p'}(\mathcal B)\ge2$ always holds.
Nielsen and Zhou stated the following conjecture
\cite[Remark~3.4]{NZ}, quoted verbatim.

\begin{conjecture}[Nielsen--Zhou, 2002]\label{conj:NZ}
The basic wavelet packets associated with a Daubechies filter of length
at least $4$ will fail to be a Schauder basis for $L^p$ when $p\ne2$,
and the functions will not be uniformly bounded in $p$-mean across
scales for any $p>2$.
\end{conjecture}

We make the second assertion precise as unboundedness of the
normalized level means
\begin{equation}\label{eq:level-mean}
 \mu_{n,p}:=\Bigl(2^{-n}\sum_{j=0}^{2^n-1}\norm{w_j}_p^p\Bigr)^{1/p},
 \qquad n\ge0.
\end{equation}
For Haar, $\mu_{n,p}=1$ for all $n$. Unboundedness of $\mu_{n,p}$ implies
the weaker statement $\sup_j\norm{w_j}_p=\infty$, so our result covers
either reading. Note that $M_p$ is the growth rate of the unnormalized
level sum; by \eqref{eq:NZ-mean},
$\mu_{n,p}^{1/n}\to2^{-1/p}M_p=4^{-1/p}\rho_p(\mathcal T)$, and the
second assertion follows from $\rho_p(\mathcal T)>4^{1/p}$.

\subsection{Main result}
The following theorem resolves both assertions of the conjecture.

\begin{theorem}\label{thm:main}
For every Daubechies filter of length $2m$, $m\ge2$, the following hold.
\begin{enumerate}[label=(\roman*)]
\item $\rho_p(\mathcal B)\rho_{p'}(\mathcal B)>2$ for every
$1<p<\infty$, $p\ne2$. Moreover, no enumeration of $\mathcal W$ is a
Schauder basis of $L^p(\R)$ for $1\le p<\infty$, $p\ne2$.
\item For every $2<p<\infty$, the limit
\[
 \gamma_{m,p}:=\lim_{n\to\infty}\norm{w_{2^n-1}}_p^{1/n}
 =2^{-1/p}\rho_p(\mathcal B)
\]
exists and is strictly greater than one.
\item For every $2<p<\infty$, the limit
\[
 \Gamma_{m,p}:=\lim_{n\to\infty}\mu_{n,p}^{1/n}
 =4^{-1/p}\rho_p(\mathcal T)
\]
exists and is strictly greater than one; equivalently,
$M_p>2^{1/p}$.
\end{enumerate}
\end{theorem}

\begin{corollary}\label{cor:conjecture}
Conjecture~\ref{conj:NZ} holds for every Daubechies filter of length
at least four.
\end{corollary}

\begin{proof}
The first assertion is Theorem~\ref{thm:main}(i); the $L^\infty$ case
is automatic because $L^\infty(\R)$ is nonseparable and hence has no
Schauder basis. By Theorem~\ref{thm:main}(iii),
$\mu_{n,p}\to\infty$ exponentially for every $2<p<\infty$.
For the endpoint interpretation, define
\[
 \mu_{n,\infty}:=\max_{0\le j<2^n}\norm{w_j}_\infty.
\]
Fix any $r\in(2,\infty)$. The support bound gives
$\norm{w_j}_r\le d^{1/r}\norm{w_j}_\infty$, hence
\[
 \mu_{n,\infty}\ge d^{-1/r}\mu_{n,r}.
\]
Thus the level maxima also grow at least exponentially; in particular,
$\sup_j\norm{w_j}_\infty=\infty$.
\end{proof}

Part~(i) is formulated without the density hypothesis of
\cite[Lemma~3.1]{NZ}, and for every enumeration of $\mathcal W$. The
rates may depend on $m$ and $p$; no uniform positive gap is asserted as
$m\to\infty$ or $p\downarrow2$. Part~(iii) is proved using the full
family $\mathcal T$, rather than inferred from a single large packet.

\subsection{Previous results}
Size properties of wavelet packets were first studied through their
frequency localization. Coifman, Meyer and Wickerhauser \cite{CMW}
showed that the level averages of $\|\widehat{w_j}\|_{L^1}$ blow up;
since the Meyer low-pass filter is nonnegative, this shows that the
Meyer wavelet packets are not uniformly bounded in $L^p(\R)$ for large
$p$. Fan \cite{Fan} determined the exponential rate of these averages,
and S\'er\'e \cite{Sere} showed that the all-high-pass branch
$w_{2^n-1}$ is in general the worst localized. Saliani \cite{Saliani}
studied the extent of wavelet packet orthonormal bases in $L^2(\R)$,
answering negatively a question of \cite{CMW} on exceptional sets of
Lebesgue measure zero.

For finite filters the positivity argument is not available. Direct
$L^p$ estimates were obtained in \cite{NRMI}, where the growth of
$w_{2^n-1}$ is expressed through the spectral radius of a subdivision
operator; numerical bounds then showed that the Daubechies,
least-asymmetric and Coiflet packet systems are not Schauder bases of
$L^p(\R)$ for large $p$. Nielsen and Zhou \cite{NZ} recast these
estimates in terms of $p$-norm joint spectral radii, leading to
\eqref{eq:NZ-branch}--\eqref{eq:NZ-condition} and
Conjecture~\ref{conj:NZ}, and showed that \eqref{eq:NZ-condition} fails
for the filter of length four when $p$ is near $1$ or $\infty$
\cite[Example~3]{NZ}.

In \cite{NACHA}, the author proved the first assertion
for the filter of length four and every $1<p<\infty$, $p\ne2$. Writing
$\rho_p(\mathcal B)=e^{P(p)/p}$ with a matrix pressure $P$, uniqueness of
matrix equilibrium states \cite{FK} shows that an affine segment of $P$
forces all periodic words to grow at the same normalized rate; an
explicit computation with the length-four coefficients then separates
two such rates. These explicit calculations do not provide a uniform
proof for all orders, and the second assertion of
Conjecture~\ref{conj:NZ} was left open.

\begin{remark}\label{rem:NRMI}
Theorem~\ref{thm:main}(ii) confirms and extends the findings of
\cite[Section~3]{NRMI}. There, numerical lower bounds for the Daubechies
filters of lengths $4$ to $40$, and for the least-asymmetric and Coiflet
filters, gave exponential growth of $\norm{w_{2^n-1}}_p$ for all
$p\ge p_0$, with an explicit $p_0>2$ depending on the filter
\cite[Corollary~3.1]{NRMI}; for the filter of length four,
$\norm{w_{2^n-1}}_p\to\infty$ was shown for every $p>2$
\cite[Theorem~3.1]{NRMI}. Theorem~\ref{thm:main}(ii) gives exponential
growth for every $p>2$ and every Daubechies filter of length at least
four and, by Proposition~\ref{prop:factors}, for the least-asymmetric filters
as well.
\end{remark}

\subsection{Strategy}
The reduction to pressure and the affine-rigidity principle are those
of \cite{NACHA}. The explicit calculations with the length-four
coefficients become increasingly cumbersome as the filter length
grows and do not provide a uniform proof for all orders. We replace
each filter-specific step by a structural argument valid for every
$m\ge2$.
\begin{itemize}
\item \emph{Irreducibility.} The $9\times9$ determinant computation of
\cite{NACHA} is replaced by Lemma~\ref{lem:irreducible}, which uses
only coprimality of $H(z)$ and $H(-z)$ and the location of the
unit-circle zeros of $H$. It is modelled on Protasov's
clean-polynomial criterion \cite{P06}.
\item \emph{Excluding affine pressure.} The explicit separation of
periodic rates, via characteristic polynomials of $B_0$ and $B_0B_1^2$,
is replaced by the theorem of Protasov and Voynov
\cite{PV}: constant spectral radius on an irreducible semigroup forces
simultaneous orthogonality of the generators, which is incompatible
with the endpoint coefficients, the energy identity, and the double
zero of the high-pass symbol (Proposition~\ref{prop:noaffine}).
\item \emph{The second assertion.} The same argument applies to the
pressure of the four-matrix family $\mathcal T$, normalized by
orthonormality at $p=2$, and yields part~(iii).
\end{itemize}
Apart from standard facts about compactly supported orthonormal
wavelets, the substantive external inputs are the theorems of Feng and
K\"aenm\"aki \cite{FK} and of Protasov and Voynov \cite{PV}. We include
the affine-rigidity argument, the packet-to-matrix transfer identities,
and the basis obstruction. Section~\ref{sec:masks} treats the masks,
Section~\ref{sec:pressure} the pressure, Section~\ref{sec:transfer}
the transfer identities, and Section~\ref{sec:conclusion} the proofs;
Section~\ref{sec:further} collects further observations.

\section{Mask identities and irreducibility}\label{sec:masks}

Retain the notation of Section~\ref{sec:intro}, with $d=2m-1$.
Both endpoints are nonzero. Set
\[
 A(z)=\sum_{k=0}^d a_kz^k,\qquad
 H(z)=A(-z)=\sum_{k=0}^d h_kz^k.
\]
The auxiliary mask $h$ is the sign-alternated low-pass mask. Reversing
it and changing its overall sign gives a translate of the conventional
high-pass mask $b$; this relation is made explicit in
Section~\ref{sec:transfer}. The quadrature mirror identities imply
the Laurent polynomial identity
\begin{equation}\label{eq:qmf}
 H(z)H(z^{-1})+H(-z)H(-z^{-1})=4\quad(z\ne0),
 \qquad \sum_{k=0}^d h_k^2=2.
\end{equation}
Indeed, the left side retains twice the even autocorrelation coefficients,
which are $2\delta_{j0}$ in this normalization.
Consequently, $H$ has no pair of opposite roots: $H(z)$ and $H(-z)$
are coprime. The point $z=0$ is excluded by $h_0\ne0$.

The Daubechies magnitude formula \cite[Section 4]{D88} gives
\begin{equation}\label{eq:magnitude}
 |H(e^{it})|^2
 =4\sin^{2m}(t/2)
   \sum_{j=0}^{m-1}\binom{m-1+j}{j}\cos^{2j}(t/2).
\end{equation}
The sum is strictly positive. Hence the only zero of $H$ on the unit
circle is $1$, and that zero has multiplicity $m$. In particular,
\begin{equation}\label{eq:double}
 H(1)=H'(1)=0.
\end{equation}

The matrices $B_0,B_1$ of \eqref{eq:matrices} are the transition
matrices of the mask $h$.

\subsection{Invertibility and irreducibility}

Invertibility and irreducibility of $B_0,B_1$ also follow from
Protasov's clean-polynomial criterion \cite[Theorem~1]{P06}, applied
to $m(z)=H(z)/2$; his setting allows an arbitrary polynomial with
nonzero endpoint coefficients, so $H(1)=0$ is no obstacle. Indeed, $H$
has no opposite roots, and it has no cycle, i.e., no finite cycle
$\lambda\mapsto\lambda^2$ of nonzero points with $H(-\lambda)=0$:
such points lie on the unit circle, hence all equal $-1$, which is not
periodic under squaring. We give direct proofs adapted to the present
setting.

We first use coprimality of the mask symbols to establish invertibility,
which will allow us to pass from invariant subspaces to invariance under
inverse transition matrices.

\begin{lemma}\label{lem:invertible}
If a polynomial $H$ of degree $d$ has nonzero constant term and
$\gcd(H(z),H(-z))=1$, both matrices in \eqref{eq:matrices} are invertible.
\end{lemma}
\begin{proof}
Identify a vector $v=(v_0,\ldots,v_{d-1})^{\mathsf T}$ with the
polynomial
\[
    V(z)=\sum_{\beta=0}^{d-1} v_\beta z^\beta,
\]
and write
\[
    H(z)V(z)=\sum_{r=0}^{2d-1} p_r z^r.
\]
The degree bound follows from $\deg H=d$ and $\deg V\leq d-1$.
For $\eta\in\{0,1\}$ and $0\leq \alpha\leq d-1$, we have
\[
    p_{\eta+2\alpha}
    =\sum_{\beta=0}^{d-1}
      h_{\eta+2\alpha-\beta}v_\beta
    =(B_\eta v)_\alpha.
\]
Thus $B_0v$ and $B_1v$ are, respectively, the complete even and odd
coefficient subsequences of $H(z)V(z)$. Here completeness follows from
$\deg(HV)\leq 2d-1$.

Suppose first that $B_0v=0$. Then every even coefficient of $H(z)V(z)$
vanishes, so $H(z)V(z)$ is an odd polynomial. Hence
\[
    H(z)V(z)=-H(-z)V(-z).
\]
If instead $B_1v=0$, then every odd coefficient vanishes, so
$H(z)V(z)$ is even and
\[
    H(z)V(z)=H(-z)V(-z).
\]
In either case,
\[
    H(z)V(z)=\pm H(-z)V(-z).
\]
It follows that $H(z)$ divides $H(-z)V(-z)$. Since
\[
    \gcd\bigl(H(z),H(-z)\bigr)=1,
\]
Euclid's lemma gives
\[
    H(z)\mid V(-z).
\]
But $\deg V<d=\deg H$, and therefore $V=0$. Consequently $v=0$.
Thus both $B_0$ and $B_1$ have trivial kernels and, being square
matrices, are invertible.
\end{proof}

To apply the matrix-pressure theorem, we must also exclude nontrivial
common invariant subspaces. The following lemma shows that the location
of the unit-circle zeros supplies this additional information.

\begin{lemma}\label{lem:irreducible}
If in addition the only unit-circle zero of $H$ is $1$, the pair
$B_0,B_1$ is irreducible over $\C$, and hence over $\R$.
\end{lemma}
\begin{proof}
It is enough to rule out a nonzero proper common invariant subspace for
$B_0^{\mathsf T}$ and $B_1^{\mathsf T}$ over $\mathbb C$. Indeed,
transposition preserves irreducibility: a common invariant subspace for
one family gives, by taking its annihilator, a common invariant subspace
for the transposed family.

Suppose, toward a contradiction, that
\[
    \{0\}\subsetneq U\subsetneq \mathbb C^d
\]
is invariant under both $B_0^{\mathsf T}$ and $B_1^{\mathsf T}$. By
Lemma~\ref{lem:invertible}, these matrices are invertible. Their restrictions to $U$ are
therefore injective and, since $U$ is finite-dimensional, surjective.
Consequently,
\[
    B_0^{\mathsf T}U=U,
    \qquad
    B_1^{\mathsf T}U=U,
\]
and $U$ is invariant under
\[
    C:=B_0^{\mathsf T}(B_1^{\mathsf T})^{-1}.
\]

For $0\leq j\leq d-2$, row $j$ of $B_0^{\mathsf T}$ equals row $j+1$
of $B_1^{\mathsf T}$. Indeed, using zero-based indices,
\[
    (B_0^{\mathsf T})_{j\beta}
    =h_{2\beta-j}
    =(B_1^{\mathsf T})_{j+1,\beta}.
\]
Since $CB_1^{\mathsf T}=B_0^{\mathsf T}$ and $B_1^{\mathsf T}$ is
invertible, the first $d-1$ rows of $C$ are
$e_1^{\mathsf T},\ldots,e_{d-1}^{\mathsf T}$. Hence
\[
    Cy=(y_1,y_2,\ldots,y_{d-1},\ell(y))^{\mathsf T}
\]
for some linear functional $\ell$ on $\mathbb C^d$.

The restriction $C|_U$ has an eigenvector because $U$ is a nonzero
finite-dimensional complex vector space. Since $C$ is invertible, the
corresponding eigenvalue $t$ is nonzero. If $0\neq y\in U$ satisfies
$Cy=ty$, comparison of the first $d-1$ coordinates gives
\[
    y_j=t^j y_0,
    \qquad 0\leq j\leq d-1.
\]
Necessarily $y_0\neq0$, and after rescaling we obtain
\[
    e(t):=(1,t,\ldots,t^{d-1})^{\mathsf T}\in U.
\]
Define
\[
    F:=\{t\in\mathbb C\setminus\{0\}:e(t)\in U\}.
\]
Then $F$ is nonempty. Any $k\le d$ vectors $e(t_1),\ldots,e(t_k)$
corresponding to distinct points are linearly independent by the
Vandermonde determinant. Since $U$ is a proper subspace of
$\mathbb C^d$, it follows that $F$ is finite and
$|F|\leq\dim U\leq d-1$.

We next record the transition identities that control square roots. For
$z\neq0$, direct calculation from the definition of $B_0$ and $B_1$
gives
\[
    B_0^{\mathsf T}e(z^2)
    =\frac12\bigl(H(z)e(z)+H(-z)e(-z)\bigr)
\]
and
\[
    B_1^{\mathsf T}e(z^2)
    =\frac{1}{2z}
      \bigl(H(z)e(z)-H(-z)e(-z)\bigr).
\]
Equivalently,
\begin{align}
    B_0^{\mathsf T}e(z^2)
    +zB_1^{\mathsf T}e(z^2)
    &=H(z)e(z),                                      \label{eq:sqrt-plus}\\
    B_0^{\mathsf T}e(z^2)
    -zB_1^{\mathsf T}e(z^2)
    &=H(-z)e(-z).                                    \label{eq:sqrt-minus}
\end{align}

Let $t\in F$ and choose $z$ with $z^2=t$. Since $e(t)\in U$ and $U$ is
invariant under $B_0^{\mathsf T}$ and $B_1^{\mathsf T}$, the left-hand
sides of \eqref{eq:sqrt-plus} and \eqref{eq:sqrt-minus} belong to $U$.
It follows that
\[
    H(z)\neq0 \quad\Longrightarrow\quad z\in F,
\]
and similarly
\[
    H(-z)\neq0 \quad\Longrightarrow\quad -z\in F.
\]
The two square roots $z$ and $-z$ cannot both be zeros of $H$, because
$\gcd(H(z),H(-z))=1$. Thus every $t\in F$ has at least one square root
that belongs to $F$ and is not a zero of $H$.

Set
\[
    G:=\{z\in F:z^2\in F,\ H(z)\neq0\}.
\]
The preceding paragraph shows that the squaring map
\[
    s:G\longrightarrow F,
    \qquad s(z)=z^2,
\]
is surjective. Since $G\subseteq F$ and $F$ is finite, we must have
$G=F$, and $s:F\to F$ is bijective. In particular, every point of $F$
is periodic under squaring. Hence, for each $z\in F$, there is an
integer $n\geq1$ such that
\[
    z^{2^n}=z.
\]
As $z\neq0$, this gives $z^{2^n-1}=1$, and therefore $|z|=1$.

Finally, fix $z\in F$. Then $z^2=s(z)\in F$, and $z$, $-z$ are the two
square roots of $z^2$. Applying the square-root implication to
$t=z^2$, we see that $H(-z)\neq0$ would give $-z\in F$. Since $z\ne0$,
the points $z$ and $-z$ are distinct elements of $F$ with the same
square, contradicting injectivity of $s:F\to F$. Thus
\[
    H(-z)=0
    \qquad\text{for every }z\in F.
\]
Every point of $F$ lies on the unit circle, and the only unit-circle
zero of $H$ is $1$. We therefore have $-z=1$, or $z=-1$, for every
$z\in F$. Hence $F\subseteq\{-1\}$. Since $F$ is nonempty,
$F=\{-1\}$. On the other hand, $s$ maps $F$ into $F$, which gives
\[
    1=(-1)^2\in F,
\]
a contradiction.

Therefore $B_0^{\mathsf T}$ and $B_1^{\mathsf T}$ have no nonzero proper
common invariant subspace over $\mathbb C$. The pair $B_0,B_1$ is thus
irreducible over $\mathbb C$, and hence also over $\mathbb R$.
\end{proof}

Together with \eqref{eq:qmf} and \eqref{eq:magnitude}, these lemmas show
that the high-pass pair is invertible and irreducible for every $m\ge2$.

\section{Affine pressure and constant spectral radius}\label{sec:pressure}

Let $\mathcal A=(A_1,\ldots,A_\ell)$ be a finite indexed family of
invertible real $D\times D$ matrices, with $\ell\ge2$. For a word
$I=i_1\cdots i_n$, write $A_I=A_{i_n}\cdots A_{i_1}$ and define
\begin{equation}\label{eq:pressure}
 P_{\mathcal A}(q)=\lim_{n\to\infty}\frac1n
 \log\sum_{|I|=n}\norm{A_I}^q,\qquad q>0,
\end{equation}
where the norm is the Euclidean operator norm. Submultiplicativity
of the sums gives existence of the limit. Bounds on the norms and
smallest singular values of the generators show that it is finite.
H\"older's inequality gives convexity. A family is irreducible if its
only common invariant real subspaces are $\{0\}$ and $\R^D$.

Let $\Sigma=\{1,\ldots,\ell\}^{\N}$ with left shift $\sigma$, and let
$\mathcal M_\sigma$ denote its invariant Borel probability measures.
For $\mu\in\mathcal M_\sigma$, set
\[
 \Lambda(\mu)=\lim_{n\to\infty}\frac1n
 \sum_{|I|=n}\mu([I])\log\norm{A_I},
\]
where $[I]$ is the cylinder with initial word $I$. This finite limit
exists by subadditivity and shift invariance. Write $h(\mu)$ for
measure-theoretic entropy. The variational principle and the following
uniqueness and Gibbs statement are the matrix-pressure input
\cite[Proposition 1.2 and (1.3)]{FK}.

\begin{theorem}[Feng--K\"aenm\"aki]\label{thm:FK}
For an irreducible family $\mathcal A$ as above and every $q>0$,
\[
 P_{\mathcal A}(q)=\sup_{\mu\in\mathcal M_\sigma}
       \{h(\mu)+q\Lambda(\mu)\}.
\]
There is a unique maximizing measure $\mu_q$. Moreover, a constant
$C_q\ge1$ exists such that, for every word $I$ of length $n$,
\begin{equation}\label{eq:gibbs}
 C_q^{-1}e^{-nP_{\mathcal A}(q)}\norm{A_I}^q
 \le\mu_q([I])\le
 C_q e^{-nP_{\mathcal A}(q)}\norm{A_I}^q.
\end{equation}
\end{theorem}

The product order in \cite{FK} is the reverse of ours. Applying that
result to $(A_1^T,\ldots,A_\ell^T)$ gives exactly
Theorem~\ref{thm:FK}: the corresponding word product is $A_I^T$, its
norm is unchanged, and transposition preserves irreducibility by
annihilators of invariant subspaces.

The following proposition turns an affine segment of the pressure into
a restriction on every finite matrix product. This is the affine-rigidity
principle of \cite{NACHA}; the proof below uses only uniqueness and the
Gibbs property.

\begin{proposition}\label{prop:rigidity}
If $\mathcal A$ is irreducible and $P_{\mathcal A}$ is affine on a
nondegenerate interval in $(0,\infty)$, then for some $\beta\in\R$,
\[
 P_{\mathcal A}(q)=\log\ell+\beta q\quad(q>0),\qquad
 \spr(A_I)=e^{\beta|I|}\quad\hbox{for every word }I.
\]
\end{proposition}
\begin{proof}
Write $P_{\mathcal A}(q)=\alpha+\beta q$ on the interval and choose
two distinct interior points $r<s$. The affine function
$h(\mu_r)+q\Lambda(\mu_r)$ lies below $P_{\mathcal A}(q)$ and agrees
with it at $r$. On the affine interval their difference is a
nonnegative affine function with an interior zero; it is identically
zero. Hence $\Lambda(\mu_r)=\beta$, $h(\mu_r)=\alpha$, and $\mu_r$
is an equilibrium state at $s$. Uniqueness gives $\mu_r=\mu_s$.
Comparing \eqref{eq:gibbs} at $r$ and $s$ yields
\[
 (C_rC_s)^{-1}e^{n(P_{\mathcal A}(s)-P_{\mathcal A}(r))}
 \le\norm{A_I}^{s-r}\le
 C_rC_s e^{n(P_{\mathcal A}(s)-P_{\mathcal A}(r))}.
\]
Division is legitimate since every product is invertible. Therefore,
with a constant $C$ independent of $I,n$,
\begin{equation}\label{eq:uniform}
 C^{-1}e^{n\beta}\le\norm{A_I}\le Ce^{n\beta}.
\end{equation}
Raising to the power $q>0$, summing over the $\ell^n$ words, and taking
logarithmic limits proves $P_{\mathcal A}(q)=\log\ell+\beta q$.
Finally, apply \eqref{eq:uniform} to the repetitions $I^k$ of a fixed
word. Since $A_{I^k}=A_I^k$, Gelfand's formula gives
$\spr(A_I)=e^{\beta|I|}$.
\end{proof}

\begin{samepage}
Once the spectral radii of all products have been forced to equal one,
the following theorem converts this information into a common
orthogonality property. We will rule out that property using the
endpoint coefficients of the filters.

\begin{theorem}[Protasov--Voynov {\cite[Theorem 2]{PV}}]
\label{thm:PV}
If a closed irreducible semigroup of real $D\times D$ matrices has
spectral radius one at every element, then in a suitable basis of
$\R^D$ all nonsingular elements of the semigroup are orthogonal matrices.
\end{theorem}
\end{samepage}

In \cite[p.~377]{PV} a family generates its semigroup by
multiplication and closure. Their definition requires the constant
spectral radius to be nonzero, which forces its value to be one. We apply the theorem to such closures;
\cite[Corollary~1]{PV} covers semigroups generated by an irreducible
family of nonsingular matrices.

\section{From cascade coefficients to packet norms}\label{sec:transfer}

We prove the scalar transfer formula of \cite[Theorem 2 and
Corollary 3.1]{NZ} in the precise finite-support convention needed here.
This also fixes the distinction between the two-generator pressure of
one branch and the four-generator pressure of the full tree.

To recover packet norms from their cascade coefficients, we first need
a comparison between the $L^p$ norm of a linear combination of translates
and the $\ell^p$ norm of its coefficients.

\begin{lemma}\label{lem:stability}
If a bounded compactly supported function $\phi$ has orthonormal
integer translates, then for every $1<p<\infty$ there are constants
$0<c_p\le C_p<\infty$ such that
\[
 c_p\norm{v}_{\ell^p}\le
 \left\|\sum_k v_k\phi(\cdot-k)\right\|_p
 \le C_p\norm{v}_{\ell^p}
\]
for every finitely supported sequence $v$.
\end{lemma}
\begin{proof}
At each point at most a fixed number of translates overlap. The
inequality $|\sum_{j=1}^N x_j|^p\le N^{p-1}\sum_j|x_j|^p$ and
integration give the upper bound. For a finitely supported $v$, put
$f=\sum_k v_k\phi(\cdot-k)$. Orthonormality, H\"older's inequality,
and the upper bound at $p'=p/(p-1)$ give, for every finitely supported $u$,
\[
 \left|\sum_k v_k\overline{u_k}\right|
 =\left|\int f\,\overline{\sum_k u_k\phi(\cdot-k)}\right|
 \le C_{p'}\norm{f}_p\norm{u}_{\ell^{p'}}.
\]
Taking the supremum proves the lower bound.
\end{proof}

For a mask $f$ supported on $\{0,\ldots,d\}$, put
\[
 T_\eta(f)_{\alpha\beta}=f(\eta+2\alpha-\beta),
 \quad0\le\alpha,\beta<d,\quad\eta\in\{0,1\}.
\]
The next lemma identifies the cascade coefficients inside products of
transition matrices. Together with the preceding stability estimate,
this gives a comparison of packet norms and matrix sums with constants
independent of the depth and the individual masks.

\begin{lemma}\label{lem:cascade}
Let $f_1,\ldots,f_n$ be masks supported on $\{0,\ldots,d\}$, and define
$u_0=\phi$ and $u_j(x)=\sum_k f_j(k)u_{j-1}(2x-k)$.
Set
\[
 \prod_{j=1}^n\left(\sum_k f_j(k)z^{2^{j-1}k}\right)
       =\sum_k q_n(k)z^k.
\]
Then, for $t=\sum_{j=1}^n2^{j-1}\eta_j$,
\begin{equation}\label{eq:entry}
 [T_{\eta_n}(f_n)\cdots T_{\eta_1}(f_1)]_{\alpha\beta}
       =q_n(t+2^n\alpha-\beta).
\end{equation}
If $\phi$ satisfies Lemma~\ref{lem:stability}, then for $1<p<\infty$,
\begin{equation}\label{eq:transfer-finite}
 \norm{u_n}_p^p\asymp_{d,p,\phi}
 2^{-n}\sum_{\eta\in\{0,1\}^n}
 \norm{T_{\eta_n}(f_n)\cdots T_{\eta_1}(f_1)}^p.
\end{equation}
\end{lemma}
\begin{proof}
The recursion gives
\[
 u_n(x)=\sum_k q_n(k)\phi(2^nx-k),\qquad
 \operatorname{supp}q_n\subseteq[0,d(2^n-1)].
\]
Formula~\eqref{eq:entry} holds at $n=1$. In the inductive step, the
sum over an intermediate index $\gamma$ has factors
\[
 f_n(\eta_n+2\alpha-\gamma)\,
 q_{n-1}(t'+2^{n-1}\gamma-\beta),
 \qquad t'=\sum_{j=1}^{n-1}2^{j-1}\eta_j.
\]
For $\gamma<0$ the second argument is negative; for $\gamma\ge d$
it exceeds $d(2^{n-1}-1)$. Thus the finite matrix sum equals the
unrestricted convolution sum. Substituting
$k=\eta_n+2\alpha-\gamma$ gives \eqref{eq:entry}.

For fixed $\beta$, division with remainder gives every integer in
$[-\beta,d2^n-1-\beta]$ a unique representation
$k=t+2^n\alpha-\beta$, with $0\le t<2^n$ and $0\le\alpha<d$.
The binary digits of $t$ are precisely $\eta_1,\ldots,\eta_n$.
Since $0\le\beta\le d-1$, the interval starts at or below zero and
ends at or above $d(2^n-1)$, so it contains the support of $q_n$. Hence
\begin{equation}\label{eq:entry-sum}
 \sum_{\eta\in\{0,1\}^n}\sum_{\alpha,\beta=0}^{d-1}
 |[T_{\eta_n}(f_n)\cdots T_{\eta_1}(f_1)]_{\alpha\beta}|^p
       =d\sum_k|q_n(k)|^p.
\end{equation}
Finite-dimensional norm equivalence compares the entry sum with the
$p$th power of the operator norm, with constants depending only on
$d,p$. Lemma~\ref{lem:stability} and a change of variables give
$\norm{u_n}_p^p\asymp_{p,\phi}2^{-n}\sum_k|q_n(k)|^p$.
Combining these facts proves \eqref{eq:transfer-finite}.
\end{proof}

To put the high-pass mask on the same support as $a$, set
\begin{equation}\label{eq:supported-high}
 g_k=(-1)^ka_{d-k}=b_{k-(d-1)}=-h_{d-k}.
\end{equation}
Replacing $b$ by $g$ translates each packet. Indeed, a translation of
a parent by $t$ translates a low-pass child by $t/2$ and a high-pass
child by $(t+d-1)/2$. Induction proves preservation of all packet
norms, even though the translation can depend on the packet.
Let $L_\eta=T_\eta(a)$ and $K_\eta=T_\eta(g)$. With $J$ the reversal
matrix on $\R^d$, direct calculation gives
\begin{equation}\label{eq:reversal}
 K_\eta=-JB_{1-\eta}J.
\end{equation}
Thus the pair $(K_0,K_1)$ has the same pressure as $(B_0,B_1)$. The
four-matrix family of the introduction is
\begin{equation}\label{eq:fourfamily}
 \mathcal T=(L_0,L_1,K_0,K_1).
\end{equation}
Write
\[
 P=P_{\mathcal B},\qquad Q=P_{\mathcal T}.
\]
Each generator is invertible by Lemma~\ref{lem:invertible} and the
QMF identity: the symbols of $a$ and $g$ also have no opposite roots.
The high-pass pair is irreducible by Lemma~\ref{lem:irreducible} and
\eqref{eq:reversal}. Every common invariant subspace for the
four-generator family is invariant under $K_0,K_1$, so irreducibility
of this pair implies irreducibility of the full family.

We can now express both packet growth rates in terms of pressure.
Evaluating these identities at $p=2$ will fix the normalizations needed
to exclude affine pressure.

\begin{corollary}\label{cor:transfer}
For every $1<p<\infty$,
\begin{align}
 \lim_{n\to\infty}\norm{w_{2^n-1}}_p^{1/n}
 &=\exp\bigl((P(p)-\log2)/p\bigr),\label{eq:branch}\\
 \lim_{n\to\infty}
 \left(2^{-n}\sum_{j=0}^{2^n-1}\norm{w_j}_p^p\right)^{1/(np)}
 &=\exp\bigl((Q(p)-\log4)/p\bigr).\label{eq:level}
\end{align}
In particular,
\begin{equation}\label{eq:normalizations}
 P(2)=\log2,\qquad Q(2)=\log4.
\end{equation}
\end{corollary}
\begin{proof}
For the branch, apply Lemma~\ref{lem:cascade} with every $f_j=g$,
and use \eqref{eq:reversal} and translation invariance of the norm.
Taking $np$th roots gives \eqref{eq:branch}.
For the full level, write $j=\sum_{r=1}^n\varepsilon_r2^{n-r}$,
with leading zeros allowed. The digits $\varepsilon_1,\ldots,\varepsilon_n$
record the successive low-pass and high-pass choices in
\eqref{eq:recursion}. Replacing $b$ by $g$ gives a translate of $w_j$.
Thus summing \eqref{eq:transfer-finite} over all $2^n$ choices
$f_r\in\{a,g\}$ sums exactly the norms of $w_0,\ldots,w_{2^n-1}$.
The constants are uniform over these choices,
and the right-hand sum is exactly the sum over all $4^n$ words in
$\mathcal T$. The dilation contributes $2^{-n}$, and the
normalization of the mean contributes another $2^{-n}$. This proves
\eqref{eq:level}. Every original packet has $L^2$ norm one, so the
left sides of both limit identities equal one at $p=2$, proving
\eqref{eq:normalizations}.
\end{proof}

Since $\rho_p(\mathcal B)=e^{P(p)/p}$ and
$\rho_p(\mathcal T)=e^{Q(p)/p}$, the identities \eqref{eq:branch} and
\eqref{eq:level} are \eqref{eq:NZ-branch} and \eqref{eq:NZ-mean} of the
introduction.

\section{Strict convexity and proof of the main theorem}\label{sec:conclusion}

The normalizations at $p=2$ force any affine pressure to be constant.
We now combine affine rigidity with the Protasov--Voynov theorem to show
that this would contradict the double zero of the high-pass symbol.

\begin{proposition}\label{prop:noaffine}
Both $P$ and $Q$ are strictly convex on $(0,\infty)$.
\end{proposition}
\begin{proof}
Suppose first that $P$ has a nondegenerate affine interval.
Proposition~\ref{prop:rigidity} gives
$P(s)=\log2+\beta s$ and $\spr(B_I)=e^{\beta|I|}$ for every word.
Since $P(2)=\log2$, the affine formula gives $\beta=0$.
Hence every finite product has spectral radius one. Set
\[
 \mathcal S_0=\{B_{i_n}\cdots B_{i_1}:n\ge1,\ i_j\in\{0,1\}\},
 \qquad \mathcal S=\overline{\mathcal S_0},
\]
with closure in the space of real $d\times d$ matrices.
Continuity of multiplication shows that $\mathcal S$ is a closed
semigroup, and continuity of the spectral radius gives $\spr(M)=1$ for
every $M\in\mathcal S$. It is irreducible because it contains the
irreducible pair $B_0,B_1$. Thus $\mathcal S$ is an irreducible real
matrix semigroup of constant spectral radius, with common value one as
in \cite[Definition~1]{PV}, and Theorem~\ref{thm:PV} applies to it: the
invertible generators $B_0,B_1\in\mathcal S$ are orthogonal in a common
basis. In particular all their eigenvalues have modulus one.

Let $e_0,\ldots,e_{d-1}$ be the standard coordinate vectors.
The first row of $B_0$ and the last row of $B_1$ give
\[
 e_0^TB_0=h_0e_0^T,\qquad e_{d-1}^TB_1=h_de_{d-1}^T.
\]
Thus $h_0,h_d$ belong to the spectra of $B_0,B_1$, respectively,
and $|h_0|=|h_d|=1$. Since $\sum_k h_k^2=2$,
all interior coefficients vanish. Thus $H(z)=h_0+h_dz^d$, whose
derivative at $1$ is $dh_d\ne0$. This contradicts \eqref{eq:double}.

If $Q$ has an affine interval, Proposition~\ref{prop:rigidity} gives
$Q(s)=\log4+\beta s$ and spectral radius $e^{\beta|I|}$ for each
word product. Since $Q(2)=\log4$, again $\beta=0$ and every finite
product has spectral radius one. As above, the closure $\mathcal S$ of
the semigroup $\mathcal S_0$ of finite products of $L_0,L_1,K_0,K_1$ is
an irreducible closed real semigroup with $\spr(M)=1$ for every
$M\in\mathcal S$. Theorem~\ref{thm:PV} makes $K_0,K_1$ orthogonal in a
common basis. Since $K_\eta=T_\eta(g)$ has the same structure as
$B_\eta$,
\[
 e_0^TK_0=g_0e_0^T,\qquad e_{d-1}^TK_1=g_de_{d-1}^T,
\]
so $g_0,g_d$ are eigenvalues of $K_0,K_1$, and $|g_0|=|g_d|=1$.
By \eqref{eq:supported-high}, this again gives $|h_0|=|h_d|=1$
and the same contradiction. Both convex pressures therefore have
no nondegenerate affine interval and hence are strictly convex:
equality at a nontrivial convex combination of two distinct points
would force affinity on the interval joining them.
\end{proof}

\begin{remark}
The same mechanism appears in \cite[Section~9.4, Proposition~14]{PV},
where constant spectral radius of an irreducible pair of refinement
transition matrices is shown to force all roots of the mask symbol onto
the unit circle. Here the endpoint coefficients and the energy identity
give a shorter contradiction with the double zero \eqref{eq:double}.
\end{remark}

\begin{samepage}
To turn the strict pressure inequality into a basis obstruction, we use
the fact that the coordinate projections of a Schauder basis are
uniformly bounded. For an orthonormal system, this gives the following
necessary condition, as used in \cite[Lemma~4.1]{NRMI} and
\cite[Lemma~3.1]{NZ}.

\begin{lemma}\label{lem:basis}
Let $(e_j)$ be an orthonormal system in $L^2(\R)$, with $e_j\in L^p(\R)\cap L^{p'}(\R)$, where $1\le p<\infty$ and
$p'$ is the conjugate exponent ($p'=\infty$ if $p=1$).
If this sequence is a Schauder basis of $L^p(\R)$, then
$\sup_j\norm{e_j}_p\norm{e_j}_{p'}<\infty$.
\end{lemma}
\end{samepage}
\begin{proof}
Let $K$ bound the partial-sum projections. Integration against $e_j$
is a continuous functional on $L^p$ and agrees on the finite span
with the $j$th coordinate functional. Density of that span identifies
the two functionals on all of $L^p$. The $j$th coordinate projection is
\[
 f\longmapsto\left(\int_{\R}f(x)\overline{e_j(x)}\,dx\right)e_j.
\]
By the norm formula for rank-one operators it has norm
$\norm{e_j}_p\norm{e_j}_{p'}$, also for $p=1$, and it is a difference of two
partial-sum projections, so its norm is at most $2K$.
\end{proof}

\begin{proof}[Proof of Theorem~\ref{thm:main}]
For $r>2$, H\"older's inequality over the $2^n$ branch matrix words gives
\[
 \sum_{|I|=n}\norm{B_I}^2
 \le 2^{n(1-2/r)}\left(\sum_{|I|=n}\norm{B_I}^r\right)^{2/r}.
\]
Taking logarithmic limits and using $P(2)=\log2$ yields
$P(r)\ge\log2$. If equality held at some $p>2$, convexity would
give $P(r)\le\log2$ for $2\le r\le p$, forcing a constant interval.
Proposition~\ref{prop:noaffine} excludes this. Hence $P(p)>\log2$
for every $p>2$, and \eqref{eq:branch} proves part~(ii).
The same argument over $4^n$ words gives $Q(p)>\log4$ for $p>2$;
\eqref{eq:level} proves part~(iii).

For part~(i), let $1<p<\infty$, $p\ne2$, and $p'=p/(p-1)$.
Since $1/p+1/p'=1$ and
\[
 \frac1p\,p+\frac1{p'}\,p'=2,
\]
strict convexity at the distinct exponents $p,p'$ gives
\[
 \frac{P(p)}p+\frac{P(p')}{p'}>P(2)=\log2,
\]
that is, $\rho_p(\mathcal B)\rho_{p'}(\mathcal B)>2$. By
\eqref{eq:branch} at the two exponents,
\begin{equation}\label{eq:product-growth}
 \lim_{n\to\infty}
 \bigl(\norm{w_{2^n-1}}_p\norm{w_{2^n-1}}_{p'}\bigr)^{1/n}
 =\frac{\rho_p(\mathcal B)\rho_{p'}(\mathcal B)}{2}>1.
\end{equation}
Each of these packets belongs to $\mathcal W$, so
Lemma~\ref{lem:basis} excludes every Schauder-basis enumeration of
$\mathcal W$ in $L^p(\R)$.

For $p=1$, fix any $r\in(1,2)$ and let $r'$ be its conjugate exponent.
H\"older's inequality gives
$\norm{f}_r\le\norm{f}_1^{1/r}\norm{f}_\infty^{1/r'}$ and
$\norm{f}_{r'}\le\norm{f}_1^{1/r'}\norm{f}_\infty^{1/r}$, hence
$\norm{f}_r\norm{f}_{r'}\le\norm{f}_1\norm{f}_\infty$. By
\eqref{eq:product-growth} at the exponent $r$,
$\norm{w_{2^n-1}}_1\norm{w_{2^n-1}}_\infty\to\infty$, and
Lemma~\ref{lem:basis} with $p=1$ excludes every Schauder-basis
enumeration of $\mathcal W$ in $L^1(\R)$.
\end{proof}

\section{Further observations}\label{sec:further}

\subsection{Other spectral factors}
The Daubechies filter of length $2m$ is the extremal-phase spectral
factor of the Daubechies polynomial; other factors, such as the
least-asymmetric filters, have the same low-pass magnitude response
\cite[Section~6.4 and Chapter~8]{D92}. Their sign-alternated symbols
$H(z)=A(-z)$ satisfy \eqref{eq:magnitude}.

\newpage
\begin{samepage}
To extend the result to other spectral factors, we must check that the
common magnitude response preserves both the mask identities and the
scaling-function properties used above. The following proposition
verifies these requirements.

\begin{proposition}\label{prop:factors}
Let $m\ge2$, and let $A(z)=\sum_{k=0}^{2m-1}a_kz^k$ be a real polynomial
of degree $2m-1$ with $A(1)=2$ and $|A(e^{i\xi})|=|A_m(e^{i\xi})|$ for
all $\xi\in\R$, where $A_m$ is the symbol of the Daubechies filter of
length $2m$; equivalently, $H(z)=A(-z)$ satisfies \eqref{eq:magnitude}.
Let $\mathcal W_A$, $\mathcal B_A$ and $\mathcal T_A$ denote the packet
system and the transition families associated with the coefficients of
$A$. Then the conclusions of Theorem~\ref{thm:main} hold for
$\mathcal W_A$, with $\mathcal B$ and $\mathcal T$ replaced by
$\mathcal B_A$ and $\mathcal T_A$. In particular, the basis-failure and
mean-growth conclusions hold for the least-asymmetric filters.
\end{proposition}
\end{samepage}

\begin{proof}
About the filter, the proof of Theorem~\ref{thm:main} in
Sections~\ref{sec:masks}--\ref{sec:conclusion} uses only the properties
verified in (a) and (b) below.

(a) \emph{Algebraic properties.} The identities \eqref{eq:lowpass} are
$A(1)=2$ and $|A(e^{i\xi})|^2+|A(-e^{i\xi})|^2=4$, which are inherited
from $A_m$. The double zero \eqref{eq:double} and the absence of further
unit-circle zeros of $H$ are determined by $|A|$ on the unit circle. For
the endpoint coefficients, the Laurent polynomials $A(z)A(z^{-1})$ and
$A_m(z)A_m(z^{-1})$ agree on the unit circle and hence identically.
Comparing coefficients of $z^{2m-1}$ gives
\[
 a_{2m-1}a_0=[z^{2m-1}]\,A_m(z)A_m(z^{-1})\ne0,
\]
so both endpoint coefficients of $A$ are nonzero.

(b) \emph{Scaling-function properties.} The scaling function $\phi$ is
supported in $[0,2m-1]$. Write
\[
 m_0(\xi)=\frac{A(e^{-i\xi})}{2}
 =\left(\frac{1+e^{-i\xi}}{2}\right)^m\mathcal L(\xi),
\]
so that $|\mathcal L(\xi)|^2=P_m(\sin^2(\xi/2))$ with
$P_m(y)=\sum_{n=0}^{m-1}\binom{m-1+n}{n}y^n$. Then
\[
 \sup_{\xi\in[-\pi,\pi]}|\mathcal L(\xi)|=\max_{y\in[0,1]}P_m(y)^{1/2}=P_m(1)^{1/2}=\binom{2m-1}{m-1}^{1/2}<2^{m-1},
\]
where the inequality follows by induction from $3<4$ at $m=2$, the
ratio of consecutive binomial coefficients being $2(2m+1)/(m+1)<4$.
By \cite[Lemma~7.1.1]{D92}, $\sup_{\xi\in[-\pi,\pi]}|\mathcal L(\xi)|<2^{m-1-\alpha}$
implies $\phi\in C^\alpha$; its proof uses only $|m_0|$, since
$|\widehat\phi(\xi)|=(2\pi)^{-1/2}\prod_{j\ge1}|m_0(2^{-j}\xi)|$. Taking
$0<\alpha<m-1-\log_2\sup_{\xi\in[-\pi,\pi]}|\mathcal L(\xi)|$ gives $\phi\in C^\alpha$,
so the compactly supported function $\phi$ is bounded. For
orthonormality we use Cohen's criterion \cite[Corollary~6.3.2]{D92}: if
the trigonometric polynomial $m_0$ satisfies $m_0(0)=1$ and
$|m_0(\xi)|^2+|m_0(\xi+\pi)|^2=1$, and has no zeros in
$[-\pi/3,\pi/3]$, then the integer translates of $\phi$ are
orthonormal. The first two conditions are \eqref{eq:lowpass}. For the
third, the common magnitude response gives
\[
 |m_0(\xi)|^2=\cos^{2m}(\xi/2)\,P_m(\sin^2(\xi/2)),
\]
and $P_m>0$ on $[0,1]$, so $m_0$ vanishes on $[-\pi,\pi]$ only at
$\pm\pi$. Since the hypothesis of the criterion only concerns $|m_0|$,
no separate analysis of invariant cycles of $\xi\mapsto2\xi$ is
needed.

(c) \emph{Conclusion.} All hypotheses used in
Sections~\ref{sec:masks}--\ref{sec:conclusion} are verified by (a) and
(b), so the proof of Theorem~\ref{thm:main} applies without change to
$\mathcal W_A$, $\mathcal B_A$ and $\mathcal T_A$.
\end{proof}

The Coiflet filters have a different modulus and are not covered by
Proposition~\ref{prop:factors}; for them the unit-circle zero structure
of $H$ would have to be checked separately.

\subsection{The role of phase}
The numerical lower bounds in \cite[Section~3]{NRMI} differ for the
Daubechies and least-asymmetric filters of the same length, suggesting
that the phase of $m_0$ influences the behaviour of the packets in
$L^p(\R)$. This is consistent with the present results: the qualitative
conclusions of Theorem~\ref{thm:main} depend only on $|m_0|$, whereas
the growth rates, which for a spectral factor $A$ we denote by
$\gamma_{A,p}$ and $\Gamma_{A,p}$, are determined by the
transition matrices, which depend on the filter coefficients
themselves, and may therefore vary with the phase.

\subsection{The Haar case}
For the Haar filter ($m=1$), $H(z)=1-z$ and the branch matrices are
the scalars $B_0=1$, $B_1=-1$. The pressure is constant, and the
double-zero obstruction fails. This is consistent with Paley's theorem
\cite{Paley} that the Walsh system is a Schauder basis of $L^p$,
$1<p<\infty$, and with $\rho_p(\mathcal T)=4^{1/p}$
\cite[Example~1]{NZ}.


\begin{thebibliography}{99}
\bibitem{CMW} R.~Coifman, Y.~Meyer, and M.~V.~Wickerhauser,
\emph{Size properties of wavelet packets}, in: M.~B.~Ruskai et al.\
(Eds.), Wavelets and Their Applications, Jones and Bartlett, Boston,
1992, pp.~453--470.
\bibitem{D88} I.~Daubechies, \emph{Orthonormal bases of compactly supported
wavelets}, Comm. Pure Appl. Math. \textbf{41} (1988), 909--996.
\href{https://doi.org/10.1002/cpa.3160410705}{doi:10.1002/cpa.3160410705}.
\bibitem{D92} I.~Daubechies, \emph{Ten Lectures on Wavelets}, SIAM,
Philadelphia, 1992.
\bibitem{Fan} A.~H.~Fan, \emph{Moyenne de localisation fr\'equentielle
des paquets d'ondelettes}, Rev. Mat. Iberoamericana \textbf{14} (1998),
63--70.
\bibitem{FK} D.-J.~Feng and A.~K\"aenm\"aki, \emph{Equilibrium states of the
pressure function for products of matrices}, Discrete Contin. Dyn. Syst.
\textbf{30} (2011), 699--708.
\bibitem{NACHA} M.~Nielsen, \emph{On a conjecture about Schauder-basis
properties of the Daubechies wavelet packets}, Appl. Comput. Harmon.
Anal. (2026), Paper No.~101938.
\href{https://doi.org/10.1016/j.acha.2026.101938}{doi:10.1016/j.acha.2026.101938}.
\bibitem{NRMI} M.~Nielsen, \emph{Size properties of wavelet packets
generated using finite filters}, Rev. Mat. Iberoamericana \textbf{18}
(2002), 249--265.
\bibitem{NZ} M.~Nielsen and D.-X.~Zhou, \emph{Mean size of wavelet packets},
Appl. Comput. Harmon. Anal. \textbf{13} (2002), 22--34.
\bibitem{Paley} R.~E.~A.~C.~Paley, \emph{A remarkable system of
orthogonal functions}, Proc. London Math. Soc. \textbf{34} (1932),
241--279.
\bibitem{P06} V.~Protasov, \emph{Refinement equations and corresponding
linear operators}, Int. J. Wavelets Multiresolut. Inf. Process.
\textbf{4} (2006), 461--474.
\href{https://doi.org/10.1142/S0219691306001385}{doi:10.1142/S0219691306001385}.
\bibitem{PV} V.~Yu.~Protasov and A.~S.~Voynov, \emph{Matrix semigroups
with constant spectral radius}, Linear Algebra Appl. \textbf{513} (2017),
376--408. \href{https://arxiv.org/abs/1407.6568}{arXiv:1407.6568}, Theorem~2.
\bibitem{Saliani} S.~Saliani, \emph{The solution of a problem of
Coifman, Meyer, and Wickerhauser on wavelet packets}, Constr. Approx.
\textbf{33} (2011), 15--39.
\bibitem{Sere} E.~S\'er\'e, \emph{Localisation fr\'equentielle des
paquets d'ondelettes}, Rev. Mat. Iberoamericana \textbf{11} (1995),
334--354.
\end{thebibliography}
\end{document}